\documentclass[10pt]{amsart}

\usepackage{geometry}
\newenvironment{nouppercase}{%
  \renewcommand{\uppercasenonmath}[1]{}}{}

\usepackage{amsmath}
\usepackage{graphicx}
\usepackage[latin1]{inputenc}
\usepackage{comment}
\usepackage{qtree,array,youngtab}

\usepackage{color, graphicx}
\usepackage{soul}
\usepackage[T1]{fontenc}

\usepackage{setspace}

\begin{spacing}{1.0}
\end{spacing}

\newcommand{\la}{\lambda}

\definecolor{darkgreen}{rgb}{0.,0.5,0.}

\numberwithin{equation}{section}
\newtheorem{thm}{Theorem}[section]
\newtheorem{lem}[thm]{Lemma}

\newtheorem{cor}[thm]{Corollary}

\newtheorem{ex}[thm]{Example}

\allowdisplaybreaks

\begin{document}
\begin{center}

{\Large \bf On the polynomiality and asymptotics of moments of sizes for random $(n, dn\pm 1)$-core partitions with distinct parts}{\let\thefootnote\relax\footnote{{$*$ This paper has been accepted for publication in SCIENCE CHINA Mathematics.}}}

\thanks{This paper has been accepted for publication in SCIENCE CHINA Mathematics.}
\end{center}

\vskip 5mm

\begin{center}
{Huan XIONG}$^{1}$ 
and {Wenston J.T. ZANG}$^{2}$ \vskip 2mm

 $^{1}$Institut de Recherche Math\'ematique Avanc\'ee, UMR 7501 \\ 
 Universit\'e de Strasbourg et CNRS,  F-67000 Strasbourg, France\\[6pt]
    $^{2}$Institute of Advanced Study of Mathematics \\
   Harbin Institute of Technology, Heilongjiang 150001, P.R. China\\[6pt]

   \vskip 2mm

Email:     $^1$xiong@math.unistra.fr, $^2$zang@hit.edu.cn 
\end{center}

\vskip 6mm \noindent {\bf Abstract.} 
Amdeberhan's conjectures on the enumeration, the average size, and the largest size of $(n,n+1)$-core partitions with distinct parts have motivated many research on this topic. Recently, Straub and Nath-Sellers obtained formulas for the numbers of $(n, dn-1)$ and $(n, dn+1)$-core partitions with distinct parts, respectively. Let $X_{s,t}$ be the size of a uniform random $(s,t)$-core partition with distinct parts when $s$ and $t$ are coprime to each other. Some explicit formulas for the $k$-th moments $\mathbb{E} [X_{n,n+1}^k]$ and $\mathbb{E} [X_{2n+1,2n+3}^k]$ were given by Zaleski and Zeilberger when $k$ is small. Zaleski also  studied the expectation and higher moments of $X_{n,dn-1}$ and conjectured some polynomiality properties concerning them in {\em arXiv:1702.05634}.

Motivated by the above works, we derive several polynomiality results and asymptotic formulas for the $k$-th moments of $X_{n,dn+1}$ and $X_{n,dn-1}$ in this paper, by studying the beta sets of core partitions. In particular, we show that these $k$-th moments are asymptotically some polynomials of n with degrees at most $2k$, when $d$ is given and $n$ tends to infinity.  Moreover, when $d=1$, we derive that the $k$-th moment $\mathbb{E} [X_{n,n+1}^k]$ of $X_{n,n+1}$ is asymptotically equal to $\left(n^2/10\right)^k$ when $n$ tends to infinity. The explicit formulas for the expectations $\mathbb{E} [X_{n,dn+1}]$ and $\mathbb{E} [X_{n,dn-1}]$  are also given. The $(n,dn-1)$-core case in our results proves several conjectures of Zaleski on the polynomiality of the expectation and higher moments of $X_{n,dn-1}$.



\section{Introduction}
A partition $\lambda$ is
called a \emph{$t$-core partition} if none of its hook lengths is divisible by~$t$. Core partitions arise naturally in the study of modular representations of finite groups. For example, they label the blocks of irreducible characters of  symmetric groups (see \cite{ols}).
Furthermore,  $\lambda$ is called a
\emph{$(t_1,t_2,\ldots, t_m)$}-core partition if it is
simultaneously  a $t_1$-core, a $t_2$-core, $\ldots$, a $t_m$-core
partition (see \cite{tamd, KN}). It is well known that, the number of $(t_1,t_2,\ldots, t_m)$-core partitions is finite if
and only if the greatest common divisor  $\gcd(t_1,t_2,\ldots, t_m)=1$ (for example, see \cite[Theorem $1$]{KN} or \cite[Theorem 1.1]{Xiong3}). 

In 2002,  Anderson \cite{and} proved the following result on the number of $(t_1,t_2)$-core
partitions, by studying  their connections with certain lattice paths.  
\begin{thm}[Anderson \cite{and}]
Let $t_1$ and $t_2$ be two coprime positive integers.
Then the number of $(t_1,t_2)$-core
partitions equals
$$  \frac{(t_1+t_2-1)!}{t_1!\, t_2!}.$$
\end{thm}
Anderson's work has  motivated many research on the enumeration, largest sizes and  average sizes of simultaneous core partitions
  (see \cite{tamd1,    CHW, EZ, ford,   N3,  SZ, Xiong1,  YZZ}). For example, when $t_1$ and $t_2$ are coprime to each other, it was proved by Olsson and Stanton
\cite{ols} that the largest size of $(t_1,t_2)$-core
partitions equals $    {(t_1^2-1)(t_2^2-1)}/{24}$, in their study of block inclusions of symmetric groups.   Armstrong (see \cite{AHJ}) gave the following conjecture on the average size of such partitions, which  was first proved by Johnson \cite{PJ} and later by Wang \cite{Wang}.

\begin{thm}[Armstrong's Conjecture]
Let $t_1$ and $t_2$ be two coprime positive integers.
Then the average size of $(t_1,t_2)$-core
partitions equals
$$  \frac{(t_1-1)(t_2-1)(t_1+t_2+1)}{24}.$$
\end{thm}

Recently, the problem on the enumeration of simultaneous core partitions with distinct parts was raised by Amdeberhan \cite{tamd}. He conjectured explicit formulas for the number, the largest size and the average size of $(n,n+1)$-core partitions with distinct parts,  which were first proved by the first author \cite{Xiong}, and later proved independently (and extended) by Straub \cite{Straub}, Nath-Sellers \cite{NS}, Zaleski \cite{Za} and Paramonov \cite{Paramonov}. 
Let $X_{s,t}$ be the size of a uniform random $(s,t)$-core partition with distinct parts when $s$ and $t$ are coprime to each other.
Zaleski \cite{Za} derived several  explicit formulas for the $k$-th moment $\mathbb{E} [X_{n,n+1}^k]$ of $X_{n,n+1}$ when $k\leq 16$.
The number, the largest size and the average size of $(2n+1,2n+3)$-core partitions with distinct parts were also well studied
  (see \cite{BNY,Paramonov,Straub,YQJZ, ZZ}). Several explicit formulas for the $k$-th (when $k\leq 7$) 
 moment $\mathbb{E} [X_{2n+1,2n+3}^k]$ of
$X_{2n+1,2n+3}$ were obtained by Zaleski and Zeilberger \cite{ZZ}. 

In 2016,  Straub \cite{Straub}  derived the following generalized Fibonacci recurrence for the number $N_d(n)$ of $(n, dn-1)$-core partitions with distinct parts.

\begin{thm}[Straub \cite{Straub}]
 Let  $N_d(1)=1$, and $N_d(n)$
be the number   of $(n, dn-1)$-core partitions with distinct parts for two positive integers  $d\geq 1$ and $n\geq 2$. Then
\begin{align}
\begin{split} \label{eq:Nds}
&N_d(1)=1, \,\, N_d(2)=d,
\\
&N_d(n)=N_d(n-1)+dN_d(n-2),\ \ \text{if}\ \  n\geq 3.
\end{split}
\end{align}
\end{thm}

The $(n, dn+1)$-core analog  was obtained later by
Nath-Sellers \cite{NS2}.

\begin{thm}[Nath-Sellers   \cite{NS2}]
Let  $M_d(-1)=0$, $M_d(0)=1$, and $M_d(n)$
be the number of $(n, dn+1)$-core partitions with distinct parts for two positive integers  $d$ and~$n$.   Then
\begin{align}
\begin{split} \label{eq:Mds}
&M_d(-1)=0, \,\, M_d(0)=1,
\\
&M_d(n)=M_d(n-1)+dM_d(n-2),\ \ \text{if}\ \  n\geq 1.
\end{split}
\end{align}
\end{thm}

Table \ref{tab:1} gives the first few values for $N_d(n)$ and $M_d(n)$.

\begin{table}[h!]
\begin{center}
\caption{The number of $(n, dn\pm 1)$-core partitions with distinct parts for $1\leq n \leq 6$.}\label{tab:1}\vspace{-2mm}
\end{center}
\begin{center}
\begin{tabular}{|l|l|l|l|l|l|l|l|l|l|} \hline
n   & 1 & 2 & 3  & 4 & 5 & 6
\\    
\hline
$N_d(n)$  & 1 & d & $2d$ & $d^2+2d$ & $3d^2+2d$ & $d^3+5d^2+2d$ \\
\hline
$M_d(n)$   & 1  & $d+1$ & $2d+1$ & $d^2+3d+1$ & $3d^2+4d+1$ & $d^3+6d^2+5d+1$  \\
\hline
\end{tabular}  
\end{center}
\end{table}
It is easy to derive that, when $d\neq 2$,
\begin{align}\label{eq:1.1}
M_d(n)=\frac{d(d-1)N_d(n)-N_d(n+1)}{d(d-2)}
\end{align}
and
\begin{align}\label{eq:1.2}
M_d(n-1)=\frac{(d-1)N_d(n+1)-d N_d(n)}{d(d-2)}.
\end{align}

Recently, the largest sizes of the above two kinds of partitions were given by the first author \cite{Xiong2}. Zaleski conjectured an explicit formula
 for the average size of $(n, dn-1)$-core partitions with distinct parts in \cite{Za2}. Furthermore, Zaleski conjectured some polynomiality properties for higher moments of their sizes.

In this paper, we derive  results on moments of sizes for random $(n, dn\pm 1)$-core partitions with distinct parts. The $(n, dn-1)$-core case proves several conjectures of Zaleski \cite{Za2}.  Let $\mathcal{C}_{n,dn+1}$ and $\mathcal{C}_{n,dn-1}$ be the sets of $(n,dn+1)$-core and $(n,dn-1)$-core partitions with distinct parts respectively. Our main results are stated next. The $(n,dn-1)$-core case in Theorems \ref{th:main1} and \ref{th:main2} are  equivalent to Zaleski's Conjectures $3.5$ and $3.1$  in  \cite{Za2}, respectively.

\begin{thm}
[see Conjecture 3.5 of Zaleski \cite{Za2}]\label{th:main1}
Let $k$ be a  positive integers.
The $k$-th power sums \begin{align*}
\sum_{\la\in \mathcal{C}_{n,dn+1}}|\la|^k
\qquad
\text{and}
\qquad
\sum_{\la\in \mathcal{C}_{n,dn-1}}|\la|^k
\end{align*} for sizes of partitions in $\mathcal{C}_{n,dn+1}$ and in $\mathcal{C}_{n,dn-1}$
   are of the form \begin{align}\label{eq:1.3}
   A(n,d) M_d(n) + B(n,d)M_d(n+1),
   \end{align}
   where $A(n,d)$ and $B(n,d)$ are some  polynomials of $n$ with degrees at most $2k$, whose coefficients are rational functions in $d$.
\end{thm}

\noindent\textbf{Remark.} In the above theorem, we use $M_d(n)$ and $ M_d(n+1) $ as a basis, while $N_d(n)$ and $ N_d(n+1) $ are used in the original statement of Zaleski's conjectures in \cite{Za2}. As mentioned by Zaleski, some of his conjectures are anomalous for the case $d=2$. The use of the basis $M_d(n)$ and $ M_d(n+1) $  avoids this problem. That is, the form \eqref{eq:1.3} always holds for any $d\geq 1$. Also, by \eqref{eq:1.1} and \eqref{eq:1.2} we know, when $d\neq 2$, $M_d(n)$ and $ M_d(n+1) $ in   \eqref{eq:1.3} can be replaced by $N_d(n)$ and $ N_d(n+1) $.

\begin{thm}
[see Conjecture 3.1 of Zaleski   \cite{Za2}]\label{th:main2}
Let $n$ and $k$ be two given positive integers.
Then the $k$-th power sums \begin{align*}
\sum_{\la\in \mathcal{C}_{n,dn+1}}|\la|^k
\qquad
\text{and}
\qquad
\sum_{\la\in \mathcal{C}_{n,dn-1}}|\la|^k
\end{align*}
are polynomials of $d$ with degrees at most $2k+\lfloor n/2\rfloor$.

\end{thm}

Recall that  $X_{n,dn-1}$ and $X_{n,dn+1}$ are sizes of uniform random $(n,dn-1)$-core and $(n,dn+1)$-core partitions with distinct parts,  respectively.
By Theorems \ref{th:main1} and \ref{th:main2} we derive the following asymptotic formulas  when $d$ is fixed or $n$ is fixed, respectively. 

\begin{thm}
\label{th:main12}
Let $d$ and $k$ be two given positive integers.
Then the $k$-th moments of $X_{n,dn+1}$ and $X_{n,dn-1}$ are asymptotically some polynomials of n with degrees at most $2k$, when $n$ tends to infinity. That is, there exist some constants $A_{d,k}$ and $B_{d,k}$ such that
\begin{align}\label{eq:main121}
\mathbb{E}[ X_{n,dn+1}^k ]= A_{d,k}\,  n^{2k}  + O(n^{2k-1})
\end{align}
and
\begin{align}\label{eq:main122}
\mathbb{E}[ X_{n,dn-1}^k ]= B_{d,k}\,  n^{2k}  + O(n^{2k-1}).
\end{align}
\end{thm}

\begin{thm}
\label{th:main22}
Let $n\geq 2$ and $k\geq 1$ be two given integers.
Then the $k$-th moments of $X_{n,dn+1}$ and $X_{n,dn-1}$ are asymptotically some  polynomials of $d$ with degrees at most $2k$,
 when $d$ tends to infinity. That is, there exist some constants $C_{n,k}$ and $D_{n,k}$ such that
\begin{align}\label{eq:main221}
\mathbb{E}[ X_{n,dn+1}^k ]= C_{n,k} \, d^{2k}  + O(d^{2k-1})
\end{align}
and
\begin{align}\label{eq:main222}
\mathbb{E}[ X_{n,dn-1}^k ]= D_{n,k}\,  d^{2k}  + O(d^{2k-1}).
\end{align}

\end{thm}

\medskip

Moreover, when $d=1$, we  derive the leading term in the asymptotic formula of $\mathbb{E}[ X_{n,n+1}^k ]$.
\begin{thm}
\label{th:main3}
Let $k$ be a given positive integer.
Then the $k$-th moment of $X_{n,n+1}$ satisfies the following asymptotic formula:
$$
\mathbb{E}[ X_{n,n+1}^k ]= \left(\frac{1}{10}\right)^k  \,  n^{2k}  + O(n^{2k-1}).
$$
\end{thm}

\medskip
We also derive explicit formulas for the expectations of $X_{n,dn+1}$  and $X_{n,dn-1}$. 
\begin{thm} \label{th:main_average_Y}
Let $d$ and $n$ be two given positive integers. The expectation of $X_{n,dn+1}$ equals
 \begin{align*}
\mathbb{E}[ X_{n,dn+1} ]&=\,  
 \frac{d(d+1)(5d+1)  (n-1)^2}{24(4d+1)}
  +
 \frac{d(d+1)(32d^2+63d+7)  (n-1)}{24(4d+1)^2}
  \\&+
 \frac{d(d+1)(6d^2+27d+3)  }{12(4d+1)^2}
  -\frac{M_d(n-1)}{M_d(n)}
  \\&
  \cdot \Bigl(\frac{d(d+1)(d-1)  (n-1)^2}{24(4d+1)}
 +  \frac{d(d+1)(14d^2+21d+1)  (n-1)}{24(4d+1)^2}
 +
 \frac{d(d+1)(6d^2+27d+3)  }{12(4d+1)^2} \Bigr)
 .
\end{align*}
\end{thm}

\medskip

\begin{ex}
Let $d=2$ and $n=4$. Then $M_d(n-1)=M_2(3)=5$ and $M_d(n)=M_2(4)=11$.  By the above theorem the expectation of $X_{n,dn+1}$ should be $54/11$. We can check that this is true since the number of $(4,9)$-core partitions with distinct parts equals $11$, and the sum of their sizes equals $54$:
$$
\mathcal{C}_{4,9}=
\{ \emptyset, (1), (2), (3), (2,1), (4,1), (5,2),(6,3),(3,2,1),(5,2,1),(4,3,2,1) \}.
$$
\end{ex}

\medskip
\begin{ex}\label{ex:1.9}
Let $d=3$ and $n=3$. Then $M_d(n-1)=M_3(2)=4$ and $M_d(n)=M_3(3)=7$.  By the above theorem the expectation of $X_{n,dn+1}$ should be $34/7$. We can check that this is true since the number of $(3,10)$-core partitions with distinct parts equals $7$, and the sum of their sizes equals $34$:
$$
\mathcal{C}_{3,10}=
\{ \emptyset, (1), (2),  (3,1), (4,2), (5,3,1),(6,4,2) \}.
$$
\end{ex}

\medskip

\begin{thm} \label{th:main_average_X}
Let $d\geq 1$ and $n\geq 2$ be two given positive integers.
The total sum of sizes of partitions in $\mathcal{C}_{n,dn-1}$ is 
 \begin{align*}
\sum_{\la\in \mathcal{C}_{n,dn-1}}|\la| & = M_d(n)\cdot
 	\Bigl( {\frac { ({d}^{2}-1)(5{d}^{2}+d-1 ) {n}^{2}}{24d(4\,d+
1)}}
-
 {\frac { (d+1)(8d^4+27d^3+2d^2-1 ) {n}}{24d(4\,d+
1)^2}}
 +
 {\frac { d^2-1   }{12d}}
 	\Bigr)
 \\&+  M_d(n-1)  \cdot
 	\Bigl(
 {\frac { ({d}+1)(-d^3+7{d}^{2}+d-1 ) {n}^{2}}{24d(4\,d+
1)}}
 -
 {\frac { (d+1)(6d^4-19d^3-7d^2+d+1 ) {n}}{24d(4\,d+
1)^2}}
\\&\ \ \ \ \ \  -
 {\frac { (d+1)(d^4+20d^3-6d^2-8d-1 )   }{12d(4\,d+
1)^2}}
 \Bigr)
 .
\end{align*}
\end{thm}

\begin{ex}
Let $d=1$ and $n=4$. Then $M_d(n-1)=M_1(3)=3$ and $M_d(n)=M_1(4)=5$.  By the above theorem the total sum of sizes of $(4,3)$-core partitions with distinct parts should be $3$. We can check that this is true since the number of such partitions  equals $N_d(n)=N_1(4)=3$, and the sum of their sizes equals $3$:
$$
\mathcal{C}_{4,3}=
\{ \emptyset, (1), (2) \}.
$$
\end{ex}

\medskip

\begin{ex}
Let $d=2$ and $n=5$. Then $M_d(n-1)=M_2(4)=11$ and $M_d(n)=M_2(5)=21$.  By the above theorem the total sum of sizes of $(5,9)$-core partitions with distinct parts should be $92$. We can check that this is true since the number of such partitions  equals $N_d(n)=N_2(5)=16$, and the sum of their sizes equals $92$:
\begin{align*}
\mathcal{C}_{5,9}=
\{ \emptyset, & (1), (2), (3),(4), (2,1), (3,1), (3,2), (5,1),(6,2),(7,3),\\& (4,2,1),(6,2,1),(4,3,1),(5,3,2), (5,4,2,1) \}.
\end{align*}
\end{ex}

By \eqref{eq:1.1} and \eqref{eq:1.2} we obtain,  
Theorem \ref{th:main_average_X} implies the following conjecture of Zaleski \cite{Za2} directly.
\begin{cor}[Conjecture 3.8 of Zaleski \cite{Za2}] \label{th:main_average}
Let $d\geq 1$ and $n\geq 2$ be two given positive integers.
When $d\neq 2$, the expectation of $X_{n,dn-1}$ equals
 \begin{align*}
\mathbb{E}[ X_{n,dn-1} ]&=\,  {\frac { \left( 5\,{d}^{3}+7\,{d}^{2}+d-1 \right) {n}^{2}}{24(4\,d+
1)}}-{\frac { \left( 8\,{d}^{5}+21\,{d}^{4}+7\,{d}^{3}-{d}^{2}+3
\,d-2 \right) n}{24(16\,{d}^{3}-24\,{d}^{2}-15\,d-2)}}
\\
&+{\frac {17\,{
d}^{4}+13\,{d}^{3}-9\,{d}^{2}-7\,d-2}{12(16\,{d}^{3}-24\,{d}^{2}-15\,d-2)}
}+\frac{N_d(n+1)}{N_d(n)}\,
\\
&\cdot \left( -{\frac { \left( {d}^{2}-1 \right) {n}^{
2}}{24(4\,d+1)}}-{\frac { \left( 2\,{d}^{4}-9\,{d}^{3}-16\,{d}^{2}-3
\,d+2 \right) n}{8(16\,{d}^{3}-24\,{d}^{2}-15\,d-2)}}-{\frac {{d}^{
4}+20\,{d}^{3}+9\,{d}^{2}-20\,d-10}{ 12\left( d-2 \right)  \left( 4\,d+1
 \right) ^{2}}} \right).
\end{align*}
\end{cor}

\medskip
The rest of  the paper is arranged as follows. In Section \ref{sec:2} we review some basic results on core partitions.  The characterizations for the $\beta$-sets of $(n, dn-1)$ and $(n, dn+1)$-core partitions with distinct parts are given in Section \ref{sec:3}. Then in Section \ref{sec:4} we use these characterizations to translate the problems to study two families of functions $G^+_{d,m,a,b}(n)$ and $G^-_{d,m,a,b}(n)$, therefore  prove the main results. The explicit formulas for expectations of $X_{n,dn+1}$ and $X_{n,dn-1}$ are derived in Section \ref{sec:5}. The asymptotic formulas for moments of $X_{n,n+1}$ are given  in Section \ref{sec:6}.

\section{Simultaneous core partitions and their $\beta$-sets}\label{sec:2}
A \emph{partition} is a finite weakly decreasing sequence of positive integers
$\lambda = (\lambda_1, \lambda_2, \ldots, \lambda_\ell)$. The numbers
 $\la_i\ (1\leq i\leq \ell)$ are called the \emph{parts} and $\sum_{1\leq i\leq \ell}\lambda_i$  the
\emph{size} of the partition $\lambda$ (see \cite{Macdonald,ec2}). Each partition $\lambda$ is identified with
 its \emph{Young
diagram}, which is an array of boxes arranged in left-justified
rows with $\lambda_i$ boxes in the $i$-th row. For the $(i,j)$-box in
the $i$-th row and $j$-th column in the Young diagram, its \emph{hook length} $h(i, j)$ is defined to be
the number of boxes exactly to the right, and exactly below, and the
box itself.  Recall that a partition $\lambda$ is
called a \emph{$(t_1,t_2,\ldots, t_m)$-core partition} if none of its hook lengths is divisible by  $t_1,t_2,\ldots,t_{m-1}$, or $t_m$  (see \cite{tamd, KN}). For example, Figure \ref{fig:1} gives the Young diagram and hook
lengths of the partition $(6,3,3,2)$. Therefore, it is a $(7, 10)$-core partition since none of its
hook lengths is divisible by $7$ or $10$.

\begin{figure}[h!]
\begin{center}
\Yvcentermath1

\begin{tabular}{c}
$\young(986321,542,431,21)$

\end{tabular}

\end{center}
\caption{The Young diagram and hook lengths of the partition
$(6,3,3,2)$.} \label{fig:1}
\end{figure}

The \emph{$\beta$-set} of a partition $\lambda = (\lambda_1, \lambda_2, \ldots, \lambda_\ell)$ is denoted by
$$\beta(\lambda)=\{\lambda_i+\ell-i : 1 \leq i \leq \ell\}.$$
In fact, $\beta(\lambda)$
 is equal to the set of
hook lengths of boxes in the first column of the corresponding Young
diagram of $\la$ (see \cite{ols,Xiong1}). For example, from Figure \ref{fig:1} we know that  $\beta((6,3,3,2))=\{  9,5,4,2 \}$.
It is easy to see that a partition~$\lambda$ is uniquely
determined by its $\beta$-set $\beta(\lambda)$.  The following results on $\beta$-sets are well known.

\begin{lem}[\cite{ols,Xiong2,Xiong,Xiong1}]\label{th:2.1}

 The size of a partition
$\lambda$ is  determined by its $\beta$-set as the following:
\begin{align}\label{eq:2.1}
 |  \lambda |=\sum_{x\in
\beta(\lambda)}{x}-\binom{|\beta(\lambda)| }{ 2}.
\end{align}
\end{lem}

\medskip

\begin{lem}[\cite{Xiong2,Xiong}] \label{th:2.2}
The partition  $\lambda$ is a partition with distinct parts if and only if there does not exist $x,y\in \beta(\lambda)$ with $x-y=1$.
\end{lem}

\medskip
\begin{lem}[\cite{and, berge,ols,Xiong,Xiong1}]\label{th:2.3}
 {\em (The abacus condition for t-core partitions.)}
A partition $\lambda$ is a $t$-core partition if and only if for any
$x\in \beta(\lambda)$ with $x\geq t$,
we always have $x-t \in \beta(\lambda)$.
\end{lem}

\section{The $\beta$-sets of $(n, dn\pm 1)$-core partitions with distinct parts} \label{sec:3}

In this section we focus on $(n, dn-1)$ and $(n, dn+1)$-core partitions with distinct parts.
The following characterizations for  $\beta$-sets are well known. We give a short proof here for completeness.

\begin{thm}[\cite{NS2,Straub,Xiong2,Za2}] \label{th: beta_dn+1}
Let $n$ and $d$  be two positive integers. Then a finite subset $S$ of $\mathbb{N}$ is a $\beta$-set of some
$(n, dn+1)$-core partition with distinct parts iff the following conditions hold:

(i) $S\subseteq \{ (i-1)n+j: 1 \leq i\leq d,\ 1\leq j \leq n-1  \}$;

(ii)  If $in+j\in S$ with $1 \leq i\leq d-1,\ 1\leq j \leq n-1$, then $(i-1)n+j\in S$;

(iii)  If $j\in S$ with $1\leq j\leq n-2$, then $j+1\notin S$.
\end{thm}

\begin{proof}
(1) Suppose that $\lambda$ is an $(n, dn+1)$-core partition with distinct parts and $S=\beta(\la)$.
By Lemma \ref{th:2.3} we have $dn+1\notin S$ and $nx\notin S$ for any $1\leq x \leq d$ since $0\notin S$.  For $x\geq dn+2$, if $x\in S,$ by Lemma \ref{th:2.3} we know $x-dn,\, x-(dn+1)\in S$. But by Lemma \ref{th:2.2} this is impossible since $\lambda$ is a partition with distinct parts. Then the condition (i) holds. Also, (ii) and (iii) hold by Lemmas \ref{th:2.2} and \ref{th:2.3}.

(2) On the other hand, suppose that the set $S$ satisfies conditions (i), (ii) and (iii). Let $\la$ be the partition with $\beta(\la)=S$. Since $\beta(\la)$ doesn't have elements larger than $dn-1$,   $\la$ must be a $(dn+1)$-core partition. Also, by (ii) $\la$ must be an $n$-core partition. Finally by (i), (ii), (iii) and Lemma \ref{th:2.2} we know $\la$ is a partition with distinct parts.
\end{proof}

Let
$$
\mathcal{A}_{d,n}:=\{(i,j): 1\leq i\leq d,\ 1\leq j\leq n  \}.
$$
We say that a subset $I\subseteq\mathcal{A}_{d,n}$ is {\em nice} if it satisfies the following two conditions:

(1) $(i+1,j)\in I$ and $i\geq 1$ imply $(i,j)\in I$; 

(2) $(1,j)\in I$ and $1\leq j\leq n-1$ imply $(1,j+1)\notin I$.

Let $\mathcal{B}^+_{d,n}$ be the set of nice subsets of $\mathcal{A}_{d,n}$.
For each $n$-core partition $\la$, define
\begin{align}\label{eq:psi}
\psi_n(\la):=\{(i,j): \ 1\leq j \leq n-1,\, (i-1)n +j \in \beta(\la)
\}.
\end{align}
Then by Theorem \ref{th: beta_dn+1} the map $\psi_n$ gives a bijection  between the sets $\mathcal{C}_{n,dn+1}$ and
 $\mathcal{B}^+_{d,n-1}$.
Furthermore, by Lemma \ref{th:2.1} we have
 \begin{lem}\label{th:size_dn+1}
\begin{align}\label{eq:size_dn+1}
\sum_{\la\in \mathcal{C}_{n,dn+1}}|\la|^k=
\sum_{I\in \mathcal{B}^+_{d,n-1}}\left(\sum_{(i,j)\in I} \left((i-1)n+j\right) - \frac{|I|^2}{2}
+ \frac{|I|}{2}
\right)^k.
\end{align}
 \end{lem}

\medskip
\begin{ex}
Let $d=3$ and $n=3$. By Example \ref{ex:1.9} we know there are $7$ of $(3,10)$-core partitions with distinct parts: $
 \emptyset, (1), (2),  (3,1), (4,2), (5,3,1),(6,4,2).
$
The corresponding nice subsets of $\mathcal{A}_{3,2}$ are:
\begin{align*}
\mathcal{B}^+_{3,2}=\{\
\emptyset,\ &\{(1,1) \},\ \{(1,2) \},\  \{(1,1),(2,1) \},\ \{(1,2),(2,2) \},
\\&
\{(1,1),(2,1),(3,1) \},\  \{(1,2),(2,2),(3,2) \}\  \}.
\end{align*}
Let $k=2$.
It is easy to check that both sides of \eqref{eq:size_dn+1} equals $282$.
\end{ex}

Similarly  the following are characterizations for $\beta$-sets of $(n, dn-1)$-core partitions with distinct parts. Notice that  $dn-1 \notin S$ in the following condition (iv).
\begin{thm}[\cite{NS2,Straub,Xiong2,Za2}]  \label{th: beta_dn-1}
Let $n\geq 2$ and $d\geq 1$    be two positive integers. Then a finite subset $S$ of $\mathbb{N}$ is a $\beta$-set of some
$(n, dn-1)$-core partition with distinct parts iff the following conditions hold:

(iv) $S\subseteq \{ (i-1)n+j: 1 \leq i\leq d,\ 1\leq j \leq n-2  \} \cup \{ in-1: 1 \leq i\leq d-1 \}$;

(v)  If $in+j\in S$ with $i\geq 1$ and $ 1\leq j \leq n-1$, then $(i-1)n+j\in S$;

(vi)  If $j\in S$ with $1\leq j\leq n-2$, then $j+1\notin S$.
\end{thm}

Let $\mathcal{B}^-_{d,n}$ be the set of nice subsets $I$ of $\mathcal{A}_{d,n}$ with $(d,n)\notin I$.
Then by Theorem \ref{th: beta_dn-1} the map $\psi_n$ defined in \eqref{eq:psi} gives a bijection  between the sets $\mathcal{C}_{n,dn-1}$ and
 $\mathcal{B}^-_{d,n-1}$.
 Furthermore, by Lemma \ref{th:2.1} we obtain
 \begin{lem}\label{th:size_dn-1}
\begin{align}\label{eq:size_dn-1}
&\sum_{\la\in \mathcal{C}_{n,dn-1}}|\la|^k=
\sum_{I\in \mathcal{B}^-_{d,n-1}}\left(\sum_{(i,j)\in I} \left((i-1)n+j\right) - \frac{|I|^2}{2}
+ \frac{|I|}{2}
\right)^k
.
\end{align}
 \end{lem}

\medskip
\begin{ex}
Let $d=3$ and $n=3$. Then there are $6$ of $(3,8)$-core partitions with distinct parts: $
 \emptyset, (1), (2),  (3,1), (4,2), (5,3,1).
$
The corresponding nice subsets of $\mathcal{A}_{3,2}$ are:
\begin{align*}
\mathcal{B}^-_{3,2}=\{\  
\emptyset,\ &\{(1,1) \},\ \{(1,2) \},\  \{(1,1),(2,1) \},\ \{(1,2),(2,2) \},
\{(1,1),(2,1),(3,1) \} \ 
\}.
\end{align*}
Let $k=2$.
Then both sides of \eqref{eq:size_dn-1} equals $138$.
\end{ex}

\section{Polynomiality of  moments of sizes for core partitions}
\label{sec:4}

In this section, we will prove the main results.

For each nice subset $I$ of $\mathcal{A}_{d,n}$, let $|I|$ be the cardinality of $I$. Define
$$
\sigma_m(I):=\sum_{(i,j)\in I} \left((i-1)m+j\right)
$$
and
$$
G^+_{d,m,a,b}(n):= \sum_{I\in \mathcal{B}^+_{d,n}}\sigma_m(I)^a \,   |I|^b
$$
for $d,m,a,b,n\geq 0$.

To compute the $k$-th power sum of sizes of  partitions in $\mathcal{C}_{n,dn+1}$, by Lemma \ref{eq:size_dn+1} we just need to compute the functions $G^+_{d,n,a,b}(n-1)$ with four variables $d,n,a,b$. The basic idea is induction on $n$. To do this, we need one more parameter $m$ here. That is, we study a more general family of functions $G^+_{d,m,a,b}(n)$ with five variables $d,m,a,b,n$.
First we derive formulas for  generating functions of~$G^+_{d,m,a,b}(n)$.  

\begin{thm}\label{th:3.1}
Assume that $a$ and $b$ are two nonnegative integers. For each $1\leq i\leq 2a+b+1$,  there exists some polynomial $P_{a,b,i}(d,m,q)$ of $d$, $m$ and $q$ with  $\deg_m (P_{a,b,i})\leq 2a+b+1-i$,  such that
the generating function of $G^+_{d,m,a,b}(n)$ equals:
\begin{align}\label{eq: G_generating}
\Psi_{d,m,a,b}:=\sum_{n\geq 0}G^+_{d,m,a,b}(n)\,  q^n=\sum_{i=1}^{2a+b+1}\frac{P_{a,b,i}(d,m,q)}{(1-q-dq^2)^{i}}.
\end{align}
\end{thm}
\begin{proof}
We will prove this result by induction on $a+b$. When $a+b=0$, we have $a=b=0$. For $n\geq 2$, 
\begin{align*}
G^+_{d,m,0,0}(n)&= \sum_{I\in \mathcal{B}^+_{d,n}}1 = |\mathcal{B}^+_{d,n}|
\\&=
\sum_{I\in \mathcal{B}^+_{d,n-1}}1+\sum_{I\in \mathcal{B}^+_{d,n}\setminus \mathcal{B}^+_{d,n-1}}1
\\&=
\sum_{I\in \mathcal{B}^+_{d,n-1}}1+\sum_{(1,n)\in I\in \mathcal{B}^+_{d,n}}1.
\end{align*}
When $(1,n)\in I\in \mathcal{B}^+_{d,n}$, we know $(1,n-1)\not\in I$ and therefore $I\cap \mathcal{A}_{d,n-1} \in \mathcal{B}^+_{d,n-2}$.
Thus for each $1\leq i\leq d$, 
$$
| \{  I\in \mathcal{B}^+_{d,n}:\  (i,n)\in I,\  (i+1,n)\notin I       \} | =   |\mathcal{B}^+_{d,n-2}|.
$$
Therefore
\begin{align}\label{eq:00_1}
G^+_{d,m,0,0}(n)&=
|\mathcal{B}^+_{d,n-1}|+d\, |\mathcal{B}^+_{d,n-2}|
=
G^+_{d,m,0,0}(n-1)+d\, G^+_{d,m,0,0}(n-2)
\end{align}
for $n\geq 2$.
By definition it is easy to derive:
\begin{align}\label{eq:00_2}
G^+_{d,m,0,0}(0)=1,\ \ \ G^+_{d,m,0,0}(1)=d+1.
\end{align}
Therefore
$$
\Psi_{d,m,0,0}-(d+1)q-1=q\, (\Psi_{d,m,0,0}-1)+d q^2\Psi_{d,m,0,0}
$$
and thus
\begin{align}\label{eq:psidm00}
\Psi_{d,m,0,0}=\frac{dq+1}{1-q-dq^2}.
\end{align}
Then the theorem is true for $a+b=0$.
Next assume that $a+b>0$ and \eqref{eq: G_generating} holds for all  pairs $(a',b')$ with $a'+b'<a+b$.  For $n\geq 2$, considering the largest integer $i$ such that $(i,n)\in I$ (or $(1,n)\notin I$), we obtain
\begin{align}\label{eq:G_recursion}
G^+_{d,m,a,b}(n)&= \sum_{I\in \mathcal{B}^+_{d,n}} \sigma_m(I)^a |I|^b 
= \sum_{(1,n)\notin  I\in \mathcal{B}^+_{d,n}} \sigma_m(I)^a |I|^b + 
\sum_{(1,n)\in  I\in \mathcal{B}^+_{d,n}} \sigma_m(I)^a |I|^b
\nonumber
\\&=
\sum_{I\in \mathcal{B}^+_{d,n-1}} \sigma_m(I)^a |I|^b +\sum_{i=1}^d \, \,    \sum\limits_{\substack{(i,n)\in  I\in \mathcal{B}^+_{d,n} \\  (i+1,n)\notin  I } }\sigma_m(I)^a |I|^b\nonumber
\\&=
\sum_{I\in \mathcal{B}^+_{d,n-1}} \sigma_m(I)^a |I|^b +\sum_{I\in \mathcal{B}^+_{d,n-2}}\sum_{i=1}^d \left(\sigma_m(I)+\binom{i}{2}m+in\right)^a (|I|+i)^b \nonumber
\\&=
G^+_{d,m,a,b}(n-1) +  d\, G^+_{d,m,a,b}(n-2)+ \sum\limits_{\substack{a'+b'< a+b\\  0\leq a'\leq a \\  0\leq b'\leq b}}\, A^{a,b}_{a',b'}(d,m,n) \,G^+_{d,m,a',b'}(n-2)
\end{align}
where 
$$
A^{a,b}_{a',b'}(d,m,n)=\binom{a}{a'}\binom{b}{b'}\sum_{i=1}^d \left(\binom{i}{2}m+in \right)^{a-a'} i^{b-b'} 
$$ 
are polynomials of $d,m,n$ such that
$$
\deg_m A^{a,b}_{a',b'} + \deg_n A^{a,b}_{a',b'} \leq a-a'.
$$

It is obvious that, when $a+b>0$,
\begin{align}\label{eq:G_generating1}
G^+_{d,m,a,b}(0)=0,\ \ \  G^+_{d,m,a,b}(1)=\sum_{i=1}^d \left(\binom{i}{2}m+i \right)^a i^b= \sum_{k=0}^a B^{a,b}_k(d)\  m^k,
\end{align}
where
$$
B^{a,b}_k(d)=\binom{a}{k}\sum_{i=1}^d \binom{i}{2}^{k} i^{a-k+b}.
$$

Considering the generating function, by \eqref{eq:G_recursion} we have
\begin{align}\label{eq:G_generating2}
\Psi_{d,m,a,b}\  - & \   q\,  G^+_{d,m,a,b}(1)  \nonumber
\\&=
q\Psi_{d,m,a,b}+dq^2\Psi_{d,m,a,b}+ \sum\limits_{\substack{a'+b'< a+b\\  0\leq a'\leq a \\  0\leq b'\leq b}}\,  \sum_{n\geq 2}\, A^{a,b}_{a',b'}(d,m,n) G^+_{d,m,a',b'}(n-2)\ q^n \nonumber
\\&=
q\Psi_{d,m,a,b}+dq^2\Psi_{d,m,a,b}+ q^2 \sum\limits_{\substack{a'+b'< a+b\\  0\leq a'\leq a \\  0\leq b'\leq b}}\, \sum_{n\geq 0}\, A^{a,b}_{a',b'}(d,m,n+2) G^+_{d,m,a',b'}(n)\ q^n.
\end{align}

When $a'+b'< a+b$,
by induction hypothesis and $$ \sum_{n\geq 0} na_n\, q^n = q\,  (\sum_{n\geq 0} a_n q^n )'$$ we obtain
\begin{align*}
\sum_{n\geq 0} n^jG^+_{d,m,a',b'}(n)\ q^n = \sum_{i=1}^{2a'+b'+1+j}\frac{C_{j,a',b',i}(d,m,q)}{(1-q-dq^2)^{i}}
\end{align*}
for each $j\geq 0$, 
where $C_{j,a',b',i}(d,m,q)$ are some polynomials  of $d$, $m$ and $q$ with
$$
\deg_m (C_{j,a',b',i}(d,m,q))\leq 2a'+b'+1+j-i.
$$
Therefore for $a'+b'< a+b$, we have 
\begin{align} \label{eq:G_generating3}
\sum_{n\geq 0}\, A^{a,b}_{a',b'}(d,m,n+2)\, G^+_{d,m,a',b'}(n)\ q^n
&= \sum_{i=1}^{2a'+b'+1+a-a'}\frac{D_{a',b',i}(d,m,q)}{(1-q-dq^2)^{i}},
\end{align}
where $D_{a',b',i}(d,m,q)$ are some polynomials  of $d$, $m$ and $q$ with
$$
\deg_m (D_{a',b',i}(d,m,q))\leq 2a'+b'+1+a-a'-i=a+a'+b'+1-i\leq 2a+b-i.
$$
Then by \eqref{eq:G_generating1}, \eqref{eq:G_generating2} and \eqref{eq:G_generating3} we obtain
\begin{align*}
\Psi_{d,m,a,b}=\sum_{n\geq 0}G^+_{d,m,a,b}(n) q^n=\sum_{i=1}^{2a+b+1}\frac{P_{a,b,i}(d,m,q)}{(1-q-dq^2)^{i}},
\end{align*}
where $P_{a,b,i}(d,m,q)$ are some polynomials of $d$, $m$ and $q$ with  \begin{align*}
\ \ \ \ \   \ \ \ \ \   \ \ \ \ \   \ \ \ \ \   \ \ \ \ \   \ \ \ \ \ 
\deg_m (P_{a,b,i}(d,m,q))\leq 2a+b+1-i.      \ \ \ \ \   \ \ \ \ \   \ \ \ \ \   \ \ \ \ \   \ \ \ \ \   \ \ \ \ \  \ \ \ \ \ \ \ \ \ \ \ \ \ \ \ \ \ \ \ \ \ \ \ \ \  \qedhere
\end{align*}
\end{proof}

\medskip
By the above theorem, to derive the explicit expression for $G^+_{d,m,a,b}(n)$, we need to study the expansion of $ 1/{(1-q-dq^2)^{k}} $.
Let $x_{d}=({1+\sqrt{1+4d}})/{2}$ and $y_{d} = ({1-\sqrt{1+4d}})/{2}$ be two roots of $x^2-x-d$.
By the partial fraction decomposition, we obtain the following results.
\begin{lem} \label{th:3.2}
Let $d$ and $k$ be given positive integers.
Then
\begin{align}\label{eq:G_generating4}
\frac{1}{(1-q-dq^2)^{k}} = \sum_{i=1}^k \, \frac{ \binom{2k-1-i}{k-1} d^{k-i}  } { {(1+4d)}^{\frac{2k-i}2}}
\, \sum_{n\geq 0}\, 	\binom{n+i-1}{i-1} \left(x_d^{n+i}+(-1)^iy_d^{n+i}\right)\, q^n.
\end{align}

\end{lem}
\begin{proof}
For $a,b\geq 0$, let
$$
F_{a,b}= \frac{1}{(1-x_{d}\, q)^a (1-y_{d}\, q)^b}.
$$
It is easy to see that
$$
F_{a+1,b+1}= \frac{x_{d}}{x_{d}-y_{d}}F_{a+1,b}+ \frac{y_{d}}{y_{d}-x_{d}}F_{a,b+1}
$$
for all $a,b\geq 0$. Therefore by induction we derive
$$
F_{a,b}= \sum_{i=1}^a \frac{ (-1)^{a-i} \binom{a+b-1-i}{b-1} \, x_d^b\, y_d^{a-i} } { {(x_d-y_d)}^{a+b-i}(1-x_{d}\, q)^i}
+
\sum_{j=1}^b \frac{ (-1)^{a} \binom{a+b-1-j}{a-1}\,   x_d^{b-j}\, y_d^{a}} { {(x_d-y_d)}^{a+b-j}(1-y_{d}\, q)^i}
$$
for all $a,b\geq 1$.
Let $a=b=k$. Then by $x_d\, y_d=-d$,  $x_d-y_d=\sqrt{1+4d}$, and
$$
\frac{1}{(1-zq)^k}= \sum_{n\geq 0}\binom{n+k-1}{k-1}z^nq^n,
$$
we derive \eqref{eq:G_generating4}.
\end{proof}

\begin{lem} \label{th:3.5}
Let $k$ be a positive integer.
Then
$$
\frac{1}{(1-q-dq^2)^{k}}=
\sum_{n\geq 0} c_n\,  q^n
$$
where $c_n$
is of the form $A(n,d) M_d(n) + B(n,d)M_d(n+1)$, such that $A(n,d)$ and $B(n,d)$ are polynomials of $n$ with degrees at most $k-1$, whose coefficients are rational functions in $d$. In particular, we have (notice that $M_d(n+2)=M_d(n+1)+dM_d(n)$)
\begin{align}\label{eq:4.3.1}
\frac{1}{1-q-dq^2} =
\sum_{n\geq 0} M_d(n)\, q^n;
\end{align}
\begin{align}\label{eq:4.3.2}
\frac{1}{(1-q-dq^2)^{2}} =
\sum_{n\geq 0} \frac{1}{4d+1} \bigl(
(n+1)\,M_d(n+2)+(n+3)\,d\, M_d(n)
\bigr)\, q^n  ;
\end{align}
and
\begin{align}\label{eq:4.3.3}
 \frac{1}{(1-q-dq^2)^{3}} &=
\sum_{n\geq 0}\Biggl(  \Bigl(\,  \frac{3d(n+1)}{(4d+1)^2}  +
\frac{1}{4d+1} \binom{n+2}{2}
\Bigr)\cdot M_d(n+2)
+
\frac{3d^2 (n+3)}{(4d+1)^2}\,  M_d(n)\Biggr)\, q^n.  
\end{align}

\end{lem}
\begin{proof}
By the recurrence relation \eqref{eq:Mds}, it is easy to see that
\begin{align}\label{eq:3.5.1}
M_d(n)=\frac{1}{\sqrt{1+4d} }\, (x_d^{n+1}-y_d^{n+1})
\end{align}
and
\begin{align}\label{eq:3.5.2}
2M_d(n+1)-M_d(n)={x_d^{n+1}+y_d^{n+1}}.
\end{align}
Therefore Lemma \ref{th:3.2} implies
$$
\frac{1}{(1-q-dq^2)^{k}}=
\sum_{n\geq 0} c_nq^n
$$
where $c_n$
is of the form $\sum_{i=0}^k A_i(n,d) M_d(n+i)$, such that each $A_i(n,d)$  is a   polynomial of $n$ with degree at most $k-1$, whose coefficients are rational functions in $d$. But by \eqref{eq:Mds}, each $M_d(n+i)$ can be written as some linear combination of $M_d(n)$ and $M_d(n+1)$,  whose coefficients are rational functions in $d$. Therefore we prove the main result of the lemma. In particular, let $k=1,2,3$ in Lemma \ref{th:3.2}, we derive \eqref{eq:4.3.1}, \eqref{eq:4.3.2} and \eqref{eq:4.3.3}.
\end{proof}

Notice that for each $i\in \mathbb{Z}$, $M_d(n+i)$ can be written as some linear combination of $M_d(n)$ and $M_d(n+1)$, whose coefficients are rational functions in $d$. 
The next result follows from Theorem \ref{th:3.1} and Lemma \ref{th:3.5} directly.

\begin{thm} \label{th:4.4}
Let $a,b\geq 0$ be some given integers. Then
$
G^+_{d,m,a,b}(n)
$
is of the form $$A(m,n,d)\,  M_d(n) + B(m,n,d)\, M_d(n+1),$$ where $A(m,n,d)$ and $B(m,n,d)$ are   polynomials of $m$ and $n$ with degrees at most $2a+b$ (that is, $\deg_m+\deg_n\leq 2a+b$), whose coefficients are rational functions in $d$.
\end{thm}
Next we give some examples of explicit expressions for $G^+_{d,m,a,b}(n)$ when $a$ and $b$ are small.

\begin{ex}

Let $a=1,\ b=0$. We have
\begin{align*}
&G^+_{d,m,1,0}(n)= \sum_{I\in \mathcal{B}^+_{d,n}}  \sigma_m(I)
=
\sum_{I\in \mathcal{B}^+_{d,n-1}} \sigma_m(I) +\sum_{I\in \mathcal{B}^+_{d,n-2}}\sum_{i=1}^d \left(\sigma_m(I)  +\binom{i}{2}m+in \right)
\\&=
G^+_{d,m,1,0}(n-1) +  d\,G^+_{d,m,1,0}(n-2)+ \left(\binom{d+1}{3} m + \binom{d+1}{2} n\right)G^+_{d,m,0,0}(n-2).
\end{align*}

Also
\begin{align*}
G^+_{d,m,1,0}(0)=0,\  G^+_{d,m,1,0}(1)=\sum_{j=1}^d\sum_{i=1}^j  (1+(i-1)m)= \binom{d+1}{3} m + \binom{d+1}{2}.
\end{align*}
Then by \eqref{eq:G_generating2} we have
\begin{align*}
\Psi_{d,m,1,0}&=  \frac{\left(\binom{d+1}{3} m + \binom{d+1}{2}\right)\, q}{1-q-dq^2} +  \frac{\left(\binom{d+1}{3} m + 2\binom{d+1}{2} \right) \,q^2 (dq+1)}{(1-q-dq^2)^2}
+  \frac{\binom{d+1}{2} q^3   ( d^2q^2+2dq+d+1 )  }{(1-q-dq^2)^3}
\\&=
   \frac{\binom{d+1}{3} m q  - \binom{d+1}{2}  q }{(1-q-dq^2)^2} +  \frac{\binom{d+1}{2}   ( 2q-q^2 )  }{(1-q-dq^2)^3}
.
\end{align*}

Therefore by \eqref{eq:Mds} and Lemma \ref{th:3.5},
\begin{align}\label{eq:dm10}
&
G^+_{d,m,1,0}(n)= \frac{1}{4d+1}
\left(\binom{d+1}{3}\, m  +
\binom{d+1}{2}\cdot \frac{n+1}{2}
\right)  n\, M_d(n)  \nonumber
\\&+
\frac{d}{4d+1}\binom{d+1}{2}
\left(  \frac{2(d-1)m}{3} +n+1
\right)  (n+1)\, M_d(n-1)
.
\end{align}
\end{ex}

\medskip

\begin{ex}
Let $a=0,\  b=1$. We have
\begin{align*}
G^+_{d,m,0,1}(n):&= \sum_{I\in \mathcal{B}^+_{d,n}}  |I|
=
\sum_{I\in \mathcal{B}^+_{d,n-1}} |I| +\sum_{I\in \mathcal{B}^+_{d,n-2}}\sum_{i=1}^d\, (|I|+i)
\\&=
G^+_{d,m,0,1}(n-1) +  d\,G^+_{d,m,0,1}(n-2)+ \binom{d+1}{2} G^+_{d,m,0,0}(n-2).
\end{align*}

Also
\begin{align*}
G^+_{d,m,0,1}(0)=0,\  G^+_{d,m,0,1}(1)=\sum_{i=1}^d  i= \binom{d+1}{2}.
\end{align*}
Then the generating function satisfies
\begin{align*}
\Psi_{d,m,0,1}-q G^+_{d,m,0,1}(1)=q\Psi_{d,m,0,1}+dq^2\Psi_{d,m,0,1}+ \binom{d+1}{2} q^2 \Psi_{d,m,0,0}.
\end{align*}
Therefore,
\begin{align*}
\Psi_{d,m,0,1}=  \frac{\binom{d+1}{2}\,q}{1-q-dq^2} +  \frac{\binom{d+1}{2}\,q^2 (dq+1) }{(1-q-dq^2)^2}.
\end{align*}
Finally,
\begin{align}\label{eq:dm01}
&
G^+_{d,m,0,1}(n)=\frac{1}{4d+1}\binom{d+1}{2}\Bigl(n\,M_d(n)
+
d(2n+2)\,M_d(n-1)
\Bigr).
\end{align}
\end{ex}

\medskip
\begin{ex} Let $a=0,\  b=2$. We have
\begin{align*}
& G^+_{d,m,0,2}(n)= \sum_{I\in \mathcal{B}^+_{d,n}}  |I|^2
=
\sum_{I\in \mathcal{B}^+_{d,n-1}} |I|^2 +\sum_{I\in \mathcal{B}^+_{d,n-2}}\sum_{i=1}^d\, (|I|+i)^2
\\&=
G^+_{d,m,0,2}(n-1) +  d\,G^+_{d,m,0,2}(n-2)+ 2\,\binom{d+1}{2} \,G^+_{d,m,0,1}(n-2)
+
\frac14\,\binom{2d+2}{3}\, G^+_{d,m,0,0}(n-2).
\end{align*}

Also
\begin{align*}
G^+_{d,m,0,2}(0)=0,\  G^+_{d,m,0,2}(1)=\sum_{i=1}^d  i^2= \frac14\binom{2d+2}{3}.
\end{align*}
Then the generating function satisfies
\begin{align*}
\Psi_{d,m,0,2}-q G^+_{d,m,0,2}(1)&=q\Psi_{d,m,0,2}+dq^2\Psi_{d,m,0,2}+ 2\binom{d+1}{2} q^2 \Psi_{d,m,0,1}
+
\frac14\binom{2d+2}{3}  q^2 \Psi_{d,m,0,0}.
\end{align*}
Therefore,
\begin{align*}
\Psi_{d,m,0,2}=
\frac{\frac14\binom{2d+2}{3}  q    }{1-q-dq^2} +  \frac{2\binom{d+1}{2}^2 q^3  +\frac14\binom{2d+2}{3}  q^2 (dq+1) }{(1-q-dq^2)^2}  +  \frac{2\binom{d+1}{2}^2 q^4 (dq+1) }{(1-q-dq^2)^3}.
\end{align*}
\end{ex}

Finally,
\begin{align}\label{eq:dm02}
G^+_{d,m,0,2}(n)
&  =
\left(\frac14\binom{2d+2}{3} \frac{1}{4d+1}-
\binom{d+1}{2}^2
\frac{6}{(4d+1)^2}
\right)\cdot n\,M_d(n)\,  \nonumber
\\&+ 
\frac14\binom{2d+2}{3} \frac{2d}{4d+1}\cdot (n+1)\,M_d(n-1)\, \nonumber
\\&+ 
\binom{d+1}{2}^2\,  \frac{1}{(4d+1)^2}\,
\left(
n^2(4d+1) +3n -4d+2
\right) \cdot M_d(n-1).
\end{align}

\medskip

Next we show that,  $G^+_{d,m,a,b}(n)$ is a polynomial of $d$ when other variables are fixed.
\begin{thm}\label{th:4.5+}
Let $m,n\geq1$ and $a,b\geq 0$ be some given integers. Then
  $G^+_{d,m,a,b}(n)$ is a polynomial of $d$ with degree $2a+b+\lfloor \frac{n+1}{2} \rfloor$.
  \end{thm}
\begin{proof}
The $a=b=0$ case is guaranteed by \eqref{eq:00_1} and \eqref{eq:00_2}. Therefore we can assume that $a+b\geq 1$.
We will prove this result by induction on $n$. It is easy to see that
\begin{align*}
G^+_{d,m,a,b}(1)&=\sum_{i=1}^d \left(\binom{i}{2}m+i\right)^a i^b,\\
G^+_{d,m,a,b}(2)&=\sum_{i=1}^d \left(\binom{i}{2}m+i\right)^a i^b + \sum_{i=1}^d \left(\binom{i}{2}m+2i\right)^a\, i^b
\end{align*}
are polynomials of $d$ with degrees $2a+b+1$,
which shows that the theorem is true for $n=1$ and $2$.

When $n\geq 3$, we assume that this result is true for $n-1$ and $n-2$. Therefore $G^+_{d,m,a,b}(n-1)$  and   $d\, G^+_{d,m,a,b}(n-2)$ are polynomials of $d$ with degrees $2a+b+\lfloor \frac{n}{2} \rfloor$ and $2a+b+\lfloor \frac{n-1}{2} \rfloor+1=2a+b+\lfloor \frac{n+1}{2} \rfloor$ respectively. Also, for $a'\leq a$, $b'\leq b$ with $a'+b'< a+ b$, we have
$$
\binom{a}{a'}\binom{b}{b'}\sum_{i=1}^d \,\left(\binom{i}{2}m+in \right)^{a-a'}\, i^{b-b'}\,  G^+_{d,m,a',b'}(n-2)
$$
is a polynomial of $d$ with degree
$$
2(a-a')+(b-b')+1+ 2a'+b'+\lfloor \frac{n-2+1}{2} \rfloor  =
2a+b+\lfloor \frac{n+1}{2} \rfloor.
$$
Therefore by \eqref{eq:G_recursion} we prove the theorem.
\end{proof}

For $a,b,m,n\geq 0$ and $d\geq 1$,
let $$
G^-_{d,m,a,b}(n):= \sum_{I\in \mathcal{B}^-_{d,n}}\sigma_m(I)^a|I|^b.
$$
Then
\begin{align*}
G^-_{d,m,0,0}(n)&
=
N_d(n+1)=M_d(n)+(d-1)M_d(n-1).
\end{align*}
When $a+b>0$, it is obvious that
\begin{align}\label{eq:G_generating1*}
G^-_{d,m,a,b}(0)=0,\ \ \  G^-_{d,m,a,b}(1)=G^+_{d-1,m,a,b}(1).
\end{align}
For $n\geq 2$, we have
\begin{align}\label{eq:G_recursion-}
G^-_{d,m,a,b}(n)&=
\sum_{I\in \mathcal{B}^+_{d,n-1}} \sigma_m(I)^a |I|^b +\sum_{I\in \mathcal{B}^+_{d,n-2}}\sum_{i=1}^{d-1}\left(\sigma_m(I)+\binom{i}{2}m+in\right)^a (|I|+i)^b \nonumber
\\&=
G^+_{d,m,a,b}(n-1) +  (d-1)G^+_{d,m,a,b}(n-2) \nonumber \\&\ \ \ + \sum_{a'+b'< a+b} B^{a,b}_{a',b'}(d,m,n) G^+_{d,m,a',b'}(n-2),
\end{align}
where $$B^{a,b}_{a',b'}(d,m,n)=\binom{a}{a'}\binom{b}{b'}\sum_{i=1}^{d-1} \left(\binom{i}{2}m+in \right)^{a-a'} i^{b-b'}.$$

Similarly as the $G^+_{d,m,a,b}(n)$ case, we obtain the following results for $G^-_{d,m,a,b}(n)$.
\begin{thm} \label{th:4.4-}
Let $a,b\geq 0$ be some given integers. Then
$
G^-_{d,m,a,b}(n)
$
is of the form $$A(m,n,d) M_d(n) + B(m,n,d)M_d(n+1),$$ where $A(m,n,d)$ and $B(m,n,d)$ are polynomials of $m$ and $n$ with degrees at most $2a+b$, whose coefficients are rational functions in $d$.
\end{thm}

\begin{thm}\label{th:4.5-}
Let $m,n\geq1$ and $a,b\geq 0$ be some given integers. Then
  $G^-_{d,m,a,b}(n)$ is a polynomial of $d$ with degree $2a+b+\lfloor \frac{n+1}{2} \rfloor$.
  \end{thm}

Now we are ready to prove the main theorems.

\begin{proof}
[Proofs of Theorems \ref{th:main1} and \ref{th:main2}]

By Lemmas \ref{th:size_dn+1} and \ref{th:size_dn-1} we know
\begin{align*}
\sum_{\la\in \mathcal{C}_{n,dn+1}}|\la|^k
\qquad
\text{and}
\qquad
\sum_{\la\in \mathcal{C}_{n,dn-1}}|\la|^k
\end{align*}
can be written as some linear combinations of $G^+_{d,n,a',b'}(n-1)$ and $G^-_{d,n,a',b'}(n-1)$ respectively,  where
$2a'+b'\leq 2k$. Notice that for each $i\in \mathbb{Z}$, $M_d(n+i)$ can be written as some linear combination of $M_d(n)$ and $M_d(n+1)$, whose coefficients are rational functions in $d$. Replace $n$ by $n-1$, and $m$ by $n$ in Theorems \ref{th:4.4} and \ref{th:4.4-}, we obtain that 
$G^+_{d,n,a',b'}(n-1)$ and $G^-_{d,n,a',b'}(n-1)$  are of the form $$A(n,d) M_d(n) + B(n,d)M_d(n+1),$$ where $A(n,d)$ and $B(n,d)$ are polynomials of $n$ with degrees $2a'+b'\leq 2k$,  whose coefficients are rational functions in $d$. Therefore Theorem \ref{th:main1} is true.
Also, Theorem \ref{th:main2} follows from  Theorems \ref{th:4.5+} and \ref{th:4.5-}.
\end{proof}

\section{Explicit formulas for expectations of $X_{n,dn+1}$ and $X_{n,dn-1}$} \label{sec:5}

In this section we give proofs of Theorems \ref{th:main_average_Y} and \ref{th:main_average_X}.

\begin{proof}
[Proof of Theorem \ref{th:main_average_Y}]
Let $k=1$ in Lemma \ref{th:size_dn+1}. We have \begin{align*}
\sum_{\la\in \mathcal{C}_{n,dn+1}}|\la|&=
\sum_{I\in \mathcal{B}^+_{d,n-1}}\left(\sum_{(i,j)\in I} \left((i-1)n+j\right) - \frac{|I|^2}{2}
+ \frac{|I|}{2} \right)
\\&= G^+_{d,n,1,0}(n-1)-\frac12 G^+_{d,n,0,2}(n-1)
+\frac12 G^+_{d,n,0,1}(n-1)
. 
\end{align*}
Then by \eqref{eq:dm10}, \eqref{eq:dm01} and \eqref{eq:dm02}  we derive
   \begin{align*}
\sum_{\la\in \mathcal{C}_{n,dn+1}}|\la|&=
M_d(n-1)\cdot \Bigl(
\frac{-d(d+1)(d-1)  (n-1)^2}{24(4d+1)}
\\&\ \ \ \ -  \frac{d(d+1)(14d^2+21d+1)  (n-1)}{24(4d+1)^2}
 -
 \frac{d(d+1)(6d^2+27d+3)  }{12(4d+1)^2}
 \Bigr)
 \\&+  M_d(n)\cdot  \Bigl(
 \frac{d(d+1)(5d+1)  (n-1)^2}{24(4d+1)}
\\&\ \ \ \  + 
 \frac{d(d+1)(32d^2+63d+7)  (n-1)}{24(4d+1)^2}
  +
 \frac{d(d+1)(6d^2+27d+3)  }{12(4d+1)^2}  \Bigr)
,
\end{align*}
which implies Theorem \ref{th:main_average_Y}.
\end{proof}

\begin{proof}
[Proof of Theorem \ref{th:main_average_X}]
Let $k=1$ in Lemma \ref{th:size_dn-1}. We have \begin{align*}
\sum_{\la\in \mathcal{C}_{n,dn-1}}|\la|&=
\sum_{I\in \mathcal{B}^-_{d,n-1}}
\left( 
\sum_{(i,j)\in I}
\left((i-1)n+j\right) - \frac{|I|^2}{2}
+ \frac{|I|}{2} \right)
\\&= G^-_{d,n,1,0}(n-1)-\frac12 G^-_{d,n,0,2}(n-1)
+\frac12 G^-_{d,n,0,1}(n-1)
.
\end{align*}
But by the definitions of $G^+_{d,m,a,b}$ and $G^-_{d,m,a,b}$ we obtain
\begin{align*}
G^-_{d,n,1,0}(n-1)&=
G^+_{d,n,1,0}(n-1)-G^+_{d,n,1,0}(n-3)- M_d(n-2)\left(\binom{d}{2}n+d(n-1)\right),
\end{align*}
\begin{align*}
G^-_{d,n,0,2}(n-1)&=
G^+_{d,n,0,2}(n-1)- \sum_{I\in \mathcal{B}^+_{d,n-3}}(|I|+d)^2
\\&= G^+_{d,n,0,2}(n-1)-G^+_{d,n,0,2}(n-3)- 2d\,  G^+_{d,n,0,1}(n-3)-d^2  M_d(n-2),
\end{align*}
and
\begin{align*}
G^-_{d,n,0,1}(n-1)&=
G^+_{d,n,0,1}(n-1)- \sum_{I\in \mathcal{B}^+_{d,n-3}}(|I|+d)
\\&= G^+_{d,n,0,1}(n-1)-G^+_{d,n,0,1}(n-3)- d  M_d(n-2).
\end{align*}
Then by \eqref{eq:dm10}, \eqref{eq:dm01} and \eqref{eq:dm02}  we derive
    Theorem \ref{th:main_average_X}.
\end{proof}

\section{Asymptotic formulas for moments of $X_{n,dn+1}$ and $X_{n,dn-1}$}  \label{sec:6}

In this section we study asymptotic behavior for moments of $X_{n,dn+1}$ and $X_{n,dn-1}$. First we give proofs of Theorems \ref{th:main12} and \ref{th:main22}. 

\begin{proof}
[Proof of Theorem \ref{th:main12}]
By the recurrence relations \eqref{eq:Nds} and \eqref{eq:Mds} it is easy to derive
\begin{align}\label{eq:5.1}
M_d(n)=\frac{1}{\sqrt{1+4d} } \left(\left(\frac{{1+\sqrt{1+4d}}}{2}\right)^{n+1}-\left(\frac{{1-\sqrt{1+4d}}}{2}\right)^{n+1}\right)
\end{align}
and
\begin{align}\label{eq:5.2}
N_d(n)=M_d(n)-M_d(n-2).
\end{align}
Then by Theorem \ref{th:main1}  we have 
$$
\mathbb{E}[ X_{n,dn+1}^k ]= A(n,d)  + B(n,d)\cdot \frac{\left(\frac{{1+\sqrt{1+4d}}}{2}\right)^{n+2}-\left(\frac{{1-\sqrt{1+4d}}}{2}\right)^{n+2}}{\left(\frac{{1+\sqrt{1+4d}}}{2}\right)^{n+1}-\left(\frac{{1-\sqrt{1+4d}}}{2}\right)^{n+1}}
$$
where $A(n,d)$ and $B(n,d)$ are some polynomials of $n$ with degrees at most $2k$. Therefore \eqref{eq:main121} holds. 
Similarly \eqref{eq:main122} follows from Theorem \ref{th:main1}, \eqref{eq:5.1} and \eqref{eq:5.2}.
\end{proof}

\begin{proof}
[Proof of Theorem \ref{th:main22}]
By the recurrence relations \eqref{eq:Nds} and \eqref{eq:Mds} it is easy to see that $M_d(n)$ and $N_d(n)$ are polynomials of of $d$ with degrees  $\lfloor n/2\rfloor$ when $n$ is given. Then Theorem \ref{th:main22} follows from Theorem \ref{th:main2}.
\end{proof}

Next we consider the asymptotic formula for $G^+_{1,0,a,b}(n)$.

\begin{thm}\label{th:5.3}
Suppose that $a$ and $b$ are two given nonnegative integers. Let $\alpha:=({1+\sqrt{5}})/{2}$.
Then
\begin{align}\label{eq:5.3.1}
G^+_{1,0,a,b}(n)= {2^{-a}\  {5}^{-(a+b+1)/2}} \  n^{2a+b}\ \alpha^{n+2-a-b}  + O(n^{2a+b-1}\alpha^{n}).
\end{align}
\end{thm}
\begin{proof}
We will prove \eqref{eq:5.3.1} by induction on $a+b$. When $a+b=0$, we have $a=b=0$.  Let $d=1$ and $m=0$ in \eqref{eq:psidm00} we derive
\begin{align}\label{eq:5.3.2}
G^+_{1,0,0,0}(n)= \frac{1}{\sqrt{5}}\, \left(\frac{{1+\sqrt{5}}}{2} \right)^{n+2} - \frac{1}{\sqrt{5}}\, \left(\frac{{1-\sqrt{5}}}{2} \right)^{n+2}={5}^{-1/2} \alpha^{n+2} + O(n^{-1}\alpha^{n}).
\end{align}
\medskip
Next assume that $a+b>0$, and  \eqref{eq:5.3.1} holds for all  pairs $(a',b')$ with $a'+b'<a+b$.

By \eqref{eq:3.5.1} and Theorem \ref{th:4.4}, for  any $a',b'\geq 0$,  there exist some constants
$C_{a',b'}$ and  $D_{a',b'}$ such that
\begin{align}\label{eq:5.3.3}
G^+_{1,0,a',b'}(n)= C_{a',b'}\  n^{2a'+b'}\ \alpha^{n}  + D_{a',b'}n^{2a'+b'-1}\ \alpha^{n} +O(n^{2a'+b'-2}\alpha^{n}).
\end{align}

Let $d=1$ and $m=0$ in \eqref{eq:G_recursion} we derive
\begin{align}\label{eq:5.3.4}
G^+_{1,0,a,b}(n)  \nonumber
&=
G^+_{1,0,a,b}(n-1) +  G^+_{1,0,a,b}(n-2)+ a\, n \,G^+_{1,0,a-1,b}(n-2)+ b \,G^+_{1,0,a,b-1}(n-2)
\\&+ \sum_{a'+b'\leq a+b-2}\, \binom{a}{a'}\binom{b}{b'}n^{a-a'} G^+_{1,0,a',b'}(n-2)
,
\end{align}
where $G^+_{d,m,a',b'}(n):=0$ if $a'<0$ or $b'<0$.
But by \eqref{eq:5.3.3}, when $a'+b'\leq a+b-2$,
$$
n^{a-a'} G^+_{1,0,a',b'}(n-2)=O(n^{2a+b-2}\alpha^{n}).
$$
Notice that $\alpha^n=\alpha^{n-1}+\alpha^{n-2}$. Also by \eqref{eq:5.3.3}, we have
\begin{align*}
G^+_{1,0,a,b}(n)  - G^+_{1,0,a,b}(n-1) -  G^+_{1,0,a,b}(n-2)=&\left( \alpha +2 \right) (2a+b) \,C_{a,b} \, n^{2a+b-1}\alpha^{n-2} 
+ O(n^{2a+b-2}\alpha^{n})
\end{align*}
and
\begin{align*}
a\, n\, G^+_{1,0,a-1,b}(n-2)+ b\, G^+_{1,0,a,b-1}(n-2)=&\left( a\,C_{a-1,b}+b\,C_{a,b-1} \right) n^{2a+b-1}\alpha^{n-2}
+ O(n^{2a+b-2}\alpha^{n}),
\end{align*}
where  $C_{a',b'}:=0$ if $a'<0$ or $b'<0$.

Therefore by \eqref{eq:5.3.4}, we have 
$$
\Bigl(\left( \alpha +2 \right) (2a+b) C_{a,b} - ( aC_{a-1,b}+bC_{a,b-1} ) \Bigr)\, n^{2a+b-1}\,\alpha^{n-2}= O(n^{2a+b-2}\,\alpha^{n}),
$$
which means that 
\begin{align}\label{eq:C_induction}
\left( \alpha +2 \right) (2a+b) C_{a,b} - ( aC_{a-1,b}+bC_{a,b-1} )=0.
\end{align}
By induction hypothesis we have 
$$
C_{a-1,\,b}=  {2^{-a+1}\  {5}^{-(a+b)/2}} \ \alpha^{3-a-b} \qquad \text{if}\ a\geq 1;
$$
and
$$
C_{a,\,b-1}=  {2^{-a}\  {5}^{-(a+b)/2}} \ \alpha^{3-a-b} \qquad \text{if}\ b\geq 1.
$$
Notice that $\sqrt{5}\alpha=\alpha+2$. Then by \eqref{eq:C_induction} we obtain
$$
C_{a,b}=  {2^{-a}\  {5}^{-(a+b+1)/2}} \ \alpha^{2-a-b}.
$$
Therefore \eqref{eq:5.3.1} holds.
\end{proof}

\medskip
Now we are ready to prove Theorem \ref{th:main3}.

\begin{proof}
[Proof of Theorem \ref{th:main3}]
By Lemma \ref{th:size_dn+1} and Theorem \ref{th:5.3} we have
\begin{align*}
\sum_{\la\in \mathcal{C}_{n,n+1}}|\la|^k
&=
\sum_{a=0}^k\, \binom{k}{a} \,\left(-\frac{1}{2}\right)^{k-a}  \,G^+_{1,0,a,2(k-a)}(n-1) + O(n^{2k-1}\alpha^{n})
\\&=
\left(
\sum_{a=0}^k\, \binom{k}{a}\, (-1)^{k-a}\,  {2^{-k}\  {5}^{-(2k-a+1)/2}} \ \alpha^{2-2k+a} \right) (n-1)^{2k}\,\alpha^{n-1} + O(n^{2k-1}\alpha^{n})
\\&=
   {2^{-k}\  {5}^{-(2k+1)/2}} \, \alpha^{2-2k}    \cdot (\sqrt{5}\alpha-1)^k \cdot n^{2k}\,\alpha^{n-1} + O(n^{2k-1}\alpha^{n}).
\end{align*}
Notice that $\sqrt{5}\alpha-1=\alpha^{2}$. Then the above formula becomes
$$
\sum_{\la\in \mathcal{C}_{n,n+1}}|\la|^k=
   {2^{-k}\  {5}^{-(2k+1)/2}} \, \alpha \cdot   n^{2k}\,\alpha^{n} + O(n^{2k-1}\alpha^{n}).
$$
Therefore 
\begin{align*}
\mathbb{E}(X_{n,n+1}^k)
&=\frac{1}{M_1(n)}
\sum_{\la\in \mathcal{C}_{n,n+1}}|\la|^k
\\ &=\frac{\sqrt{5}}{ \alpha^{n+1}-(1-\alpha)^{n+1}}  \cdot  \Bigl(
{2^{-k}\  {5}^{-(2k+1)/2}} \, \alpha \cdot   n^{2k}\,\alpha^{n} + O(n^{2k-1}\alpha^{n}) \Bigr)
\\ &=\left(\frac{1}{10}\right)^k  \,  n^{2k} + O(n^{2k-1}).
\end{align*}
\end{proof}

\section{Further Directions}
We derive several polynomiality results and asymptotic formulas for moments of sizes of random $(n, dn\pm 1)$-core partitions with distinct parts, which prove several conjectures of Zaleski \cite{Za2}. In the past few years, the numbers, the largest sizes and  the average sizes of $(n, n+1)$, $(2n+1, 2n+3)$-core partitions with distinct parts were also well studied by many mathematicians (see \cite{BNY,NS,Paramonov,Straub,Xiong,YQJZ,ZZ,Za}). But for general $(s, t)$-core partitions with distinct parts, even for the $(n, n+3)$-core case, we know very little. We hope that the methods used  and results obtained in this paper provide some clues for studying the general $(s, t)$-core case.

Also, Zaleski \cite[Conjecture 3.4]{Za2} conjectured that the distribution of $(n, dn-1)$-core partitions with distinct parts is asymptotically normal as $n$ tends to infinity when $d$ is given.  At this moment, we are unable to prove this asymptotic distribution conjecture. By the idea from Zeilberger \cite{Z1}, to try to prove this conjecture, we need to have a better understanding of the leading terms in the asymptotic formulas of $\mathbb{E}[ X_{n,dn+1}^k ]$ and $\mathbb{E}[ X_{n,dn-1}^k ]$, 
which means that we should study the coefficients of the generating functions in \eqref{eq: G_generating}.  It would be interesting to find a proof of this distribution  conjecture and furthermore study the distribution of general $(s, t)$-core partitions with distinct parts.

\section*{Acknowledgments}
The first author acknowledges support from the Swiss National Science Foundation (Grant number P2ZHP2\_171879). This work was done during the first author's visit to the Harbin Institute of Technology (HIT). The first author would like to thank  Prof. Quanhua Xu and the second author for the hospitality. The authors really appreciate the valuable suggestions given by referees for improving the overall quality of the manuscript.



\begin{thebibliography}{1}



\bibitem{tamd}
 Amdeberhan T. Theorems, problems and conjectures.   ArXiv:1207.4045,  2012 

 \bibitem{tamd1}
Amdeberhan T,    Leven E. Multi-cores, posets, and lattice paths.  Adv in Appl Math,  2015, 71: 1--13

\bibitem{and}
Anderson J. Partitions which are simultaneously $t_1$- and
$t_2$-core. Disc Math, 2002, 248: 237--243

\bibitem{AHJ}
Armstrong D,  Hanusa C R H,   Jones B. Results and conjectures on simultaneous core partitions.
European J Combin, 2014, 41:  205--220

\bibitem{BNY}
Baek J, Nam H, Yu M. A bijective proof of Amdeberhan's conjecture on the number of $(s,s+2)$-core partitions with distinct parts.  Disc Math,  2018, 341(5): 1294--1300


\bibitem{berge}  Berge C, \emph{Principles of Combinatorics}, 
Mathematics in Science and Engineering Vol. 72.  New York: Academic Press,  1971


\bibitem{CHW}
Chen W, Huang H,   Wang L. Average size of a self-conjugate $s,t$-core partition.  Proc Amer Math Soc,  2016, 144(4): 1391--1399

\bibitem{EZ}
Ekhad S B,  Zeilberger D. Explicit expressions for the variance and higher moments of the size of a simultaneous core partition and its limiting distribution, in The Personal Journal of Shalosh B. Ekhad and Doron Zeilberger,  [2015-08-30],  available at:
\url{http://www.math.rutgers.edu/$\sim$zeilberg/mamarim/mamarimhtml/stcore.html}

\bibitem{ford}
 Ford B,  Mai H,   Sze L. Self-conjugate simultaneous $p$- and
$q$-core partitions and blocks of $A_n$. J Number Theory,
2009, 129(4): 858--865




\bibitem{PJ}
 Johnson P, Lattice points and simultaneous core partitions.  Electron J Combin, 2018, 25(3): Paper 3.47
 

 

\bibitem{KN}
Keith W J,  Nath R.  Partitions with prescribed hooksets. J
Comb Number Theory, 2011, 3(1):   39--50

\bibitem{Macdonald}  Macdonald I G. \emph{Symmetric Functions and Hall Polynomials}, Oxford Mathematical Monographs, second edition. New York: The Clarendon Press, Oxford University Press, 1995

\bibitem{NS}
Nath R,   Sellers J A. A combinatorial proof of a relationship between largest $(2k-1,2k+1)$ and $(2k-1,2k,2k+1)$-cores. Electron J Combin, 2016, 23(1): Paper 1.13


\bibitem{NS2}
Nath R,   Sellers J A.
Abaci structures of $(s, ms\pm1) $-core partitions. 
Electron J Combin, 2017, 24(1): Paper 1.5



\bibitem{N3}
Nath R. Symmetry in largest $(s-1,s+1)$ cores. Integers, 2016, 16: Paper No. A18



\bibitem{ols}
Olsson J,  Stanton D. Block inclusions and cores of partitions.
Aequationes Math, 2007, 74(1-2): 90--110

\bibitem{Paramonov}
Paramonov K.
Cores with distinct parts and bigraded Fibonacci numbers. 
Disc Math, 2018,  341(4):  875--888



\bibitem{ec2} 
Stanley R P. \emph{Enumerative Combinatorics}, vol.~2. New York/Cam\-bridge: Cambridge University Press,  1999.

\bibitem{SZ}
 Stanley R P,  Zanello F. The Catalan case of Armstrong's conjectures on simultaneous core partitions. SIAM J Discrete Math, 2015, 29(1): 658--666



\bibitem{Straub}
Straub A. Core partitions into distinct parts and an analog of Euler's theorem. European J Combin, 2016, 57: 40--49

\bibitem{Wang}
 Wang V. Simultaneous core partitions: parameterizations and sums.  
Electron J Combin, 2016, 23(1): Paper 1.4



\bibitem{Xiong2}
 Xiong H. On the largest sizes of certain simultaneous core partitions with distinct parts. European J Combin, 2018, 71: 33--42

\bibitem{Xiong}  Xiong H. Core partitions with distinct parts. 
Electron J Combin, 2018, 25(1): Paper 1.57

\bibitem{Xiong1}
 Xiong H. On the largest size of $(t,t+1,..., t+p)$-core partitions. Disc Math, 2016,  339(1): 308--317

\bibitem{Xiong3}
Xiong H. The number of simultaneous core partitions.    ArXiv:1409.7038,  2014

\bibitem{YQJZ}
Yan S,  Qin G,  Jin Z,  Zhou R. On $(2k+ 1, 2k+ 3)$-core partitions with distinct parts.  Disc Math, 2017, 340(6):  1191--1202

\bibitem{YZZ}
Yang J,  Zhong M, Zhou R. On the enumeration of $(s,s+1,s+2)$-core partitions. European J Combin, 2015, 49: 203--217

\bibitem{ZZ}
Zaleski A,  Zeilberger D. Explicit expressions for the expectation, variance and higher moments of the size of a $(2n+ 1, 2n+ 3)$-core partition with distinct parts. J Difference Equ Appl, 2017, 23(7):    1241--1254

\bibitem{Za}
Zaleski A. 
Explicit expressions for the moments of the size of an $(s,s+1)$-core partition with distinct parts.
Adv in Appl Math, 2017, 84: 1--7

\bibitem{Za2}
Zaleski A. 
Explicit expressions for the moments of the size of an $(n, dn-1)$-core partition with distinct parts.  ArXiv:1702.05634,  2017 



\bibitem{Z1}
Zeilberger D. The automatic central limit theorems generator (and much more!), in: Kotsireas I,  Zima E. Eds. Advances in Combinatorial Mathematics: Proceedings of the Waterloo Workshop in Computer Algebra 2008. Berlin: Springer Verlag, 2010,  165--174




\end{thebibliography}
\end{document}